\documentclass[12pt,a4paper]{amsart}
\usepackage{amsfonts, amsmath, amssymb, amscd, amsthm, array}

\usepackage{multicol}
\usepackage{subfig}
\usepackage{color}
\usepackage[all]{xy}
\usepackage{hyperref,url}

\makeatletter
\def\blfootnote{\xdef\@thefnmark{}\@footnotetext}
\makeatother

\newtheorem{thm}{Theorem}[section]

\newtheorem{lem}[thm]{Lemma}
\newtheorem{prop}[thm]{Proposition}

\theoremstyle{definition}

\newtheorem{df}[thm]{Definition}

\theoremstyle{remark}

\newcommand{\gen}[1]{\left\langle#1\right\rangle}

\begin{document}

\title{A note on the Farrell-Jones conjecture for relatively hyperbolic groups}
\author{Yago Antol\'{i}n}
\address[Yago Antol\'{i}n]{1326 Stevenson Center,
Vanderbilt University,
Nashville, TN 37240, USA.}
\email{yago.anpi@gmail.com}
\urladdr{https://sites.google.com/site/yagoanpi/}

\author{Giovanni Gandini}
\address[Giovanni Gandini]{K{\o}benhavns Universitet, Institut for Matematiske Fag, Universitetsparken 5, 2100 K{\o}benhavn \O, Denmark}
\email{ggandini@math.ku.dk}
\urladdr{http://www.math.ku.dk/~zjb179}

\thanks{The work of the first author was partially supported by  the MCI (Spain) through project MTM2011-25955. The second  author was supported by the Danish National Research Foundation (DNRF) through the Centre for Symmetry and Deformation}


\begin{abstract} 
For a  group $G$ relatively hyperbolic to a family of 
residually finite groups satisfying the Farrell-Jones conjecture, we reduce the solution of the Farrell-Jones conjecture for $G$ to the case of certain nice cyclic extensions in $G$.

\end{abstract}


\maketitle

\setcounter{tocdepth}{1}


\section{Introduction}
Let $\mathfrak{C}$ be the class of groups satisfying  the K- and L-theoretic Farrell-Jones Conjecture with finite wreath products (with coefficients in additive categories) with respect to the family of virtually cyclic subgroups. The statement  of the Farrell-Jones conjecture and its applications can be found in 
\cite{bartels2011farrell, nilp}. 
Let $G$ be a  group relatively hyperbolic to a family of 
residually finite groups lying in $\mathfrak{C}$. We remark that from several known results on the Farrell-Jones conjecture and a theorem of  Dahmani, Guirardel and Osin  it follows that $G$ lies in $\mathfrak{C}$ if  certain ``nice''  cyclic extensions lie in $\mathfrak{C}$.

\section{Relatively hyperbolic groups} 
In this section we recall some basic results. 
We follow Osin's definition, which doesn't require the group to be finitely generated. See \cite{Osin06} for details. 

\begin{df}\label{def:RH}
A group $G$ is {\it relatively hyperbolic with respect to a family of subgroups $\{H _\omega \}_{\omega \in \Omega}$}, if it admits a finite  presentation relative to $\{H _\omega \}_{\omega \in \Omega}$ 
and this presentation has a linear relative Dehn function. 

The subgroups $\{H_\omega\}_{\omega \in \Omega}$ are called  {\it peripheral
(or parabolic) subgroups } of $G.$
An element $g\in G$ is called {\it loxodromic} if it has infinite order and it is not conjugate to any element of a parabolic subgroup.
\end{df}


\begin{lem}\cite[Theorem 1.4]{Osin06}\label{lem:A}
Suppose that $G$ is relative hyperbolic with respect to $\{H_\omega\}_{\omega \in \Omega}$. 

\begin{enumerate}
\item[{\rm (i)}] For  any $\omega,\lambda\in \Omega$,  $\omega\neq \lambda$ and $g,h\in G$,  $|H_{\omega}^{g}\cap H_{\lambda}^{h}|<\infty$.
\item[{\rm (ii)}] 
For  any $\omega\in \Omega$,  and $g\in G-H_\omega$, $|H_{\omega}^{g}\cap H_{\omega}|<\infty$.
\end{enumerate}
\end{lem}
A family of subgroups  $\{H_\omega\}_{\omega \in \Omega}$ of $G$ is {\it almost malnormal} if it
satisfy the conclusions of Lemma \ref{lem:A}.

The following is a simplification of \cite[Theorem 2.40]{Osin06}.
\begin{lem}\label{lem:hyperbolic_parabolic}
Suppose that $G$ is relative hyperbolic with respect to $\{H_\omega\}_{\omega \in \Omega}\cup\{H\}$ where $H$ is a  Gromov hyperbolic group, then $G$ is 
relatively hyperbolic with respect to $\{H_\omega\}_{\omega \in \Omega}$.
\end{lem}

Finally our main tool is the following theorem, which is a simplification of \cite[Theorem 7.19]{DGO}.
We remark that the first part of the next theorem was already obtained by Osin in \cite{O07}.

\begin{thm}\label{thm:DF}
Let $G$ be a group hyperbolic relatively to $\{H_\omega\}_{\omega\in \Omega}$.

For any finite subset $Y$ of $G$ there exists  a collection of finite subsets $\Phi_\omega \subset G\setminus\{1\}$, $\omega\in \Omega$, such that for any collection of normal subgroups $N_\omega\trianglelefteq H_\omega$ satisfying that $N_\omega$ avoids $\Phi_\omega$,
the quotient $G/\langle\langle  \cup_\omega N_\omega\rangle\rangle$ is hyperbolic relative to $\{H_\omega/N_\omega\}_{\omega\in \Omega}$, and the quotient map, restricted to $Y$, is injective.

Moreover $\langle\langle  \cup_\omega N_\omega\rangle\rangle$ is a free
product of conjugates of $N_\omega$'s. 
\end{thm}

We observe that a bit more can be said about $\langle\langle  \cup_i N_i\rangle\rangle$.
\begin{lem}\label{lem:trans}
Let $G$ be a group and $\{H_\omega\}_{\omega\in \Omega}$  a family of almost malnormal subgroups. Suppose that for $\omega\in \Omega$ there is $N_\omega\trianglelefteq H_\omega$ such that $N_\omega$ is not a  free product of free groups and finite groups.
Suppose that for each $\omega\in \Omega$ there is a subset $S_{\omega}$ in $G$ such that  $$K=\langle\langle  \cup_i N_i\rangle\rangle=*_{\omega\in \Omega}(*_{s\in S_\omega} N_\omega^s),$$
then, $S_\omega$ is a transversal for the $(H_\omega K)$-action on $G$ on the left.
\end{lem}
\begin{proof}
Let $s,r \in S_\omega$ and suppose that
$hks=r$ for some $h\in H_\omega$, $k\in K$. Then
$N_\omega^{hks}=N_\omega^{r}$. But $N_\omega^{hks}=N_\omega^{ks}=N_\omega^{sk'}$ for some $k'\in K$.
Thus $N_\omega^s$ and $N_\omega^r$ are two free factors
in a free product decomposition of $K$, conjugated by
$k'\in K$ and hence $s=r$.

Let $g\in G$.   By the Kurosh subgroup theorem (see \cite[I.7.8]{DD89}) $N_\omega^g \leqslant K$ is a
free product where the free factors are either infinite cyclic or of the form $A_{\lambda,r,k}=N_\omega^{g}\cap N_\lambda^{rk}$, where $\lambda\in \Omega$ and $r\in S_\lambda$ and $k\in K$. If all the $A_{\lambda,r, k}$ are finite, $N_\omega^g$ is a
free product of finite and infinite cyclic groups and 
this contradicts our assumptions. So, 
$$|N_\omega^g \cap N_\lambda^{rk}|=\infty$$ for some 
$\lambda\in \Omega$, $r\in S_\lambda$. By almost 
malnormality, we conclude that 
$\lambda=\omega$ and $g\in H_\omega r k$, and hence $g\in H_\omega K r$.
\end{proof}

\section{Relatively hyperbolic groups in the class \texorpdfstring{$\mathfrak{C}$}{C}}
We collect now some properties of the class $\mathfrak{C}$.

\begin{prop}\label{prop:clos}\cite{bartels2008k, bartels2011farrell, bartels2012k} The following properties hold:
\begin{enumerate}
\item[{\rm (1)}] $\mathfrak{C}$ is closed under taking subgroups.
\item[{\rm (2)}] $\mathfrak{C}$ is closed under free products.
\item[{\rm (3)}] Gromov hyperbolic groups and abelian groups are in $\mathfrak{C}$.
\item[{\rm (4)}] If $\pi\colon G\to H$ is a morphism such that $H$ is in $\mathfrak{C}$
and for every infinite cyclic group $Z$ of $H$, $\pi^{-1}(Z)$ is in $\mathfrak{C}$ then $G$ is in $\mathfrak{C}$.
\end{enumerate}
\end{prop}

The objective of this note is establish the following observation, that reduces the Farrell-Jones conjecture 
for relatively hyperbolic groups with respect to residually finite groups to a particular kind of group
extensions.
\begin{prop}\label{main}
Let $G$ be  group relatively hyperbolic to a family of 
residually finite finitely generated groups 
$\{H_\omega\}_{\omega\in \Omega}$ lying in $\mathfrak{C}$. 

 There exists a family of normal subgroups 
$N_\omega \trianglelefteq H_\omega$ such that 
$$K=\langle\langle \cup N_\omega\rangle \rangle =*_{\omega\in \Omega}*_{s\in S_\omega} N_\omega^s,$$ 
 $G/K$ is hyperbolic and each $S_\omega$ is a transveral for the $(H_\omega K)$-action of $G$ on the left. In particular $K$ and $G/K$ are in $\mathfrak{C}$.

Moreover, the following are equivalent.
\begin{enumerate}
\item[(a)] $G$ is in $\mathfrak{C}$.
\item[(b)] for every $g\in G$ loxodromic element,
$\gen{K,g}$ is in $\mathfrak{C}$.
\end{enumerate}

\end{prop}

\begin{proof}
We first notice that we can assume that no $H_\omega$ is hyperbolic. Indeed, if some $H_\omega$ is hyperbolic, we can put $N_\omega=\{1\}$ and   use Lemma \ref{lem:hyperbolic_parabolic} to remove this group from the family of parabolic subgroups.

Take $Y=\emptyset \subseteq G$  and for $\omega\in \Omega$, let $\Phi_\omega\subset G-\{1\}$ be the finite set provided by
Theorem \ref{thm:DF}. Using that each $H_\omega$ is  residually finite, we
can find subgroups $N_\omega\trianglelefteq H_\omega$  that $N_\omega$ avoids $\Phi_\omega$ such that $H_\omega/N_\omega$ is finite (and in particular hyperbolic). Thus, by Theorem \ref{thm:DF}, there exists a quotient map $\pi\colon G\to Q$, where $Q$
is relatively hyperbolic with respect to $\{H_\omega/N_\omega\}_{\omega\in \Omega}$, and 
$K=\ker (\pi)=*_{\omega\in \Omega} *_{s\in S_\omega} N_\omega^s $, where the $S_\omega's$ are 
subsets of $G$. Since each $H_\omega/N_\omega$ is hyperbolic,
Lemma \ref{lem:hyperbolic_parabolic}, implies that $Q$
is hyperbolic. By Propostion \ref{prop:clos}(1)--(3), $K$ and $G/K$ are in $\mathfrak{C}$.

Since $H_\omega$ is finitely generated and not hyperbolic, $N_\omega$ is finitely generated and not hyperbolic.
Hence,  $N_\omega$ is not a free product of finite and free groups. The family $\{H_\omega\}_{\omega\in \Omega}$ is almost malnormal by \ref{lem:A} and hence all the hypothesis of Lemma \ref{lem:trans}
are satisfied and then  each $S_\omega$ is
a transversal for the $(H_\omega K)$-action on $G$ on the left.

To establish the moreover part, first notice that, by Proposition \ref{prop:clos} (2),  (a) implies (b).

Assume that (b) holds. We note that since each $N_\omega$ has finite index in $H_\omega$, every element in $H_\omega$ is mapped to a finite order element of $Q$. If $q$ is an infinite order element of $Q$, then any $g\in G$, satisfying that $\pi(g)=q$, is a loxodromic element of $G$. By (b), $\gen{K,g}=\pi^{-1}(\gen{q})$ is in $\mathfrak{C}$ and then by Proposition \ref{prop:clos} (4),  $G$ is in $\mathfrak{C}$.
\end{proof}

\smallskip
\noindent {\bf Acknowledgements.} The authors would like to thank Henrik R\"uping for pointing out a mistake in a previous version.

\end{document}